\documentclass[a4paper, 11pt]{amsart}

\usepackage[T1]{fontenc}
\usepackage[utf8]{inputenc}
\usepackage{geometry}
\usepackage{amssymb}

\makeatletter
\@namedef{subjclassname@2020}{%
  \textup{2020} Mathematics Subject Classification}
\makeatother

\allowdisplaybreaks[2]

\newtheorem{theorem}{Theorem}[section]

\newtheorem{lemma}[theorem]{Lemma}

\theoremstyle{definition}
\newtheorem{remark}[theorem]{Remark}

\newtheorem{examples}[theorem]{Examples}
\newtheorem{example}[theorem]{Example}

\numberwithin{equation}{section}

\newcommand\QQ{\mathbb{Q}}
\newcommand\RR{\mathbb{R}}
\newcommand\NN{\mathbb{N}}
\newcommand\CC{\mathbb{C}}
\newcommand\ZZ{\mathbb{Z}}
\begin{document}

\title[The Wedderburn-Artin Theorem]{The Wedderburn-Artin Theorem}

\author{Matej Brešar} 

\address{Faculty of Mathematics and Physics, University of Ljubljana \&
Faculty of Natural Sciences and Mathematics, University of Maribor \& IMFM, Ljubljana, Slovenia}
\email{matej.bresar@fmf.uni-lj.si}

\thanks{Partially supported by ARIS Grant P1-0288}


\begin{abstract} 
The celebrated Wedderburn-Artin theorem states that 
a simple left artinian ring is isomorphic to the ring of matrices over a division ring.
    We give a short and self-contained proof
which avoids the use of modules. 
\end{abstract}


\maketitle

\section{Introduction}

 The following theorem can be considered the fundamental theorem of noncommutative algebra.

\begin{theorem}\label{AW} {\bf (Wedderburn-Artin Theorem)} If $R$ is a simple left artinian ring, then there exist a division ring $D$ and a positive integer $n$ such that $R\cong M_n(D)$.
\end{theorem}

Joseph Wedderburn proved Theorem \ref{AW} for finite-dimensional algebras in 1908 \cite{W}. Some twenty years later, Emil Artin  generalized 
Wedderburn's seminal result to rings satisfying chain conditions \cite{A}.

We will present a proof that requires only
 basic knowledge of undergraduate algebra and does not involve the concept of a module. The standard 
module-theoretic approach is certainly efficient and yields further insight into the structure theory of rings, but  requires more prerequisites and may be  a bit difficult to understand for a beginner.

This paper is an adaptation and revision of the author's earlier paper  \cite{B} in which a proof of  Wedderburn's classical version of Theorem \ref{AW} is given. This proof is also presented in the book \cite{Bb}. The proof in the present paper is based on the same idea, but is more transparent and covers the general case.


In Section \ref{s2}, we give all necessary  definitions, provide  basic examples, and prove several lemmas.  
The proof of Theorem \ref{AW}, in fact of a  slightly more general version of this theorem, is given in
  Section \ref{s3}.

\section{Preliminaries}\label{s2}

We assume the reader is familiar with the concept of a (not necessarily commutative) ring.
Our rings will be assumed to be
 unital, i.e., they contain an element $1$, called a unity, that satisfies
$x=1x=x1$ for every ring element $x$. We remark that, like in \cite{B}, the existence of $1$ could be avoided,
 but at the cost of simplicity of proofs.

 To refresh the reader's memory, we start by recalling some basic definitions and facts.

An element $y$ is a left inverse of the element $x$ if $y x=1$. 
A right inverse $z$ is defined analogously. If $x$ has both  a left inverse $y$ and a right inverse $z$, then  $y=y(xz)=(yx)z=z$. A left inverse of $x$ thus
coincides with a right inverse of $x$. 
We call it the inverse of $x$ and denote it by $x^{-1}$. An element having an inverse is said to be invertible. 

A ring is  nonzero if  $0$ is not its only element (equivalently, $0\ne 1$).
A nonzero ring in which every nonzero element is invertible  is called a  division ring. 
Such a ring may not be commutative. For example, the ring of quaternions $\mathbb H$, with which  the reader is presumably familiar, is an example of a noncommutative division ring.
A commutative division ring is called a field. Standard examples are the fields  of rational numbers $\QQ$,
 real numbers $\RR$, and  complex numbers $\CC$.

A (ring) isomorphism is bijective map $\varphi$ from a ring $R$ to a ring $R'$ that satisfies $\varphi(a+b)=\varphi(a)+\varphi(b)$ and $\varphi(ab)=\varphi(a)\varphi(b)$ for all $a,b\in R$. We say that $R$ and $R'$ are isomorphic if there exists an isomorphism from $R$ to $R'$. In this case, we write $R\cong R'$.

An additive subgroup $I$ of a ring $R$ is called
a left ideal (resp. right ideal) if $u x\in I$ (resp. $xu\in I$) for all $u \in I$ and $x\in R$. If $I$ is both a left and a right ideal, then it is called  an ideal. In commutative rings, left and right ideals are automatically ideals.

We  introduce some notation.
For a subset  $S$ 
of a ring $R$ and  elements $a,b\in R$,  write $Sa= \{sa\,|\,s\in S\}$, $aS=\{as\,|\,s\in S\}$, and
$aSb=\{asb\,|\,s\in S\}$. 
By $RS$ we denote 
the set of all sums of elements of the form $xs$ with $x\in R$ and $s\in S$.
Note that $RS$ is the
left ideal of $R$ generated by $S$ (i.e., the smallest left ideal containing $S$). Taking $S=aR$, we obtain  $RaR$ which is the ideal of $R$ generated by $a$.
If $I$ and $J$ are left ideals of $R$, then $IJ$ denotes the set of all sums of elements $uv$ with $u\in I$ and $v\in J$. Observe that $IJ$ is again a left ideal (and is an ideal if $I$ and $J$ are ideals). If $I=J$,  we write $I^2$ for $IJ$.

In the next subsections, we introduce several notions that are not  always treated in  basic algebra courses.

\subsection{Simple Rings}

 A nonzero ring $R$ whose only ideals are $\{0\}$ and $R$ is called a {\bf simple ring}.

\begin{examples}\label{ex21} (a)
Obvious examples of simple rings are division rings, which  do not even contain left (or right) ideals different from  $\{0\}$ and $R$.
This is because if a left ideal $I$ of  any ring $R$ contains an invertible element $a$, then  $I=R$. Indeed, $x=(xa^{-1})a\in I$ for 
every $x\in R$.

(b)   Let  $D$ be a ring and $n\in \NN$. The set
   of all $n\times n$ matrices with entries in $D$ is a ring under the usual matrix addition and multiplication.
   We denote it by $M_n(D)$ (for $n=1$, this is just $D$). 
 Let $E_{ij}$ denote the matrix whose $(i,j)$ entry 
  is $1$ and all other entries are $0$. We call $E_{ij}$, $1\le i,j\le n$,  the {\bf standard matrix units}. 
By $aE_{ij}$ we denote the matrix whose $(i,j)$ entry is $a\in D$ and  other entries are $0$. 
Observe that for every  $A=(a_{ij})\in M_n(D)$, \begin{equation*}\label{emu}
E_{ij}AE_{kl} = a_{jk}E_{il}\quad\mbox{for all $1\le i,j,k,l\le n$.} 
\end{equation*}
Now assume  that $D$ is a division ring. 
Let $I$ be a nonzero ideal of $M_n(D)$. Take a nonzero $ (a_{ij})\in I$ and choose $j$ and $k$ such that $a_{jk}\ne 0$. From the displayed formula we see that $a_{jk}E_{il}\in I$ for all $i$ and $l$, and hence also
 $(d a_{jk}^{-1})E_{ii}\cdot a_{jk}E_{il} = dE_{il}\in I$ for every $d\in D$. Since every matrix in $M_n(D)$ is a sum of matrices of the form $dE_{il}$, $I=M_n(D)$.
We have thus proved that  $M_n(D)$ is a simple ring for every $n\in\NN$ and every division ring $D$. 
\end{examples}

\begin{remark}
If $a\ne 0$ is an element of a ring $R$ that does not 
 have a left inverse, then the left ideal  $Ra$ does not contain $1$ and is therefore different from $\{0\}$ and $R$. This  readily implies that division rings are actually the only rings  without proper nonzero left ideals. In particular, the  ring $M_n(D)$ has proper nonzero left ideals for every $n\ge 2$.
\end{remark}

 \subsection{Prime Rings}

A ring $R$ is called a {\bf prime ring}
if for all ideals $I$ and $J$ of $R$,
$IJ=\{0\}$ implies $I=\{0\}$ or $J=\{0\}$. Obviously, every simple ring is prime, and so is every ring without zero-divisors (i.e., a ring in which $ab=0$ implies $a=0$ or $b=0$). The ring of integers $\ZZ$ is thus prime,  but  not simple since $n\ZZ$ is an ideal of $\ZZ$ for every $n\in\NN$. 
   
   \begin{lemma}\label{defprime}
    Let $R$ be a ring. The following conditions  are equivalent:
 \begin{enumerate}
 
   \item[{\rm (i)}] $R$ is prime.
   \item[{\rm (ii)}] For all $a,b\in R$, $aRb =\{0\}$ implies $a=0$ or $b=0$.
 \item[{\rm (iii)}] For all left ideals $I$ and $J$ of $R$, $IJ=\{0\}$ implies $I=\{0\}$ or $J=\{0\}$.
  \end{enumerate}
   \end{lemma}
   
   \begin{proof} (i)$\implies$(ii). 
Observe that
$aRb=\{0\}$ yields $(RaR)(RbR)=\{0\}$, which by assumption implies 
$RaR=\{0\}$ and hence $a=0$, or $RbR=\{0\}$ and hence $b=0$.

(ii)$\implies$(iii). If left ideals $I$ and $J$ satisfy $IJ=\{0\}$, then
$aRb=\{0\}$ for every $a\in I$ and every $b\in J$. Therefore, by assumption, $a=0$ or $b=0$.

   (iii)$\implies$(i). Trivial.
   \end{proof}

\subsection{Left Artinian Rings}
A ring $R$ is said to be 
an  {\bf algebra} over a field $F$ if $R$  is also a vector space over $F$ and $\lambda(xy)=(\lambda x)y = x(\lambda y)$ holds for all $\lambda\in F$ and $x,y\in R$. A finite-dimensional algebra is  an algebra which is finite-dimensional as a vector space.

\begin{example} \label{eprvi}
    The ring $M_n(F)$ becomes a finite-dimensional  algebra if we endow it with the usual scalar multiplication $\lambda(a_{ij})=(\lambda a_{ij})$. The standard matrix units $E_{ij}$ form its basis, so its dimension is $n^2$.
\end{example}

The notion of a left artinian ring is a generalization of the notion of a finite-dimensional algebra. It is defined as follows: $R$ is  a {\bf left artinian ring}  (or is said to satisfy the {\bf descending chain condition} on left ideals) if 
for every descending chain $L_1\supseteq L_2\supseteq L_3\supseteq\dots$ of left ideals of $R$ there exists an $n_0\in\NN$ such that $L_n=L_{n_0}$ for every $n\ge n_0$.

\begin{examples} \label{eprvia} (a) Since a left ideal of an algebra is a  vector subspace (because if $\lambda\in F$
and $u\in L$ then $\lambda u = (\lambda 1)u\in L$), a finite-dimensional algebra is indeed left-artinian. 

(b) Another obvious example of a left artinian ring is every division ring.
Indeed,  we saw above that such a ring does not even have proper nonzero left ideals. We claim that, more generally, the matrix  ring $M_n(D)$ is  left artinian for every $n\in \NN$ and  every division ring $D$. If $D$ is commutative,   this follows from (a) and Example  \ref{eprvi}. 
It turns out that essentially the same proof  works if $D$ is not commutative. We omit details, but only indicate why this is true.
 One  defines a vector space over a division ring in exactly the same way as a vector space over a field (but we call it more precisely a left vector space), and shows 
that every such vector space has a basis and that all bases have the same cardinality. This enables one
to define the concept of dimension in the usual way. Now, the ring $M_n(D)$ is a vector space of dimension $n^2$ over $D$, and its left ideals are automatically vector subspaces. This implies that $M_n(D)$ is left artinian. (Incidentally, we do  not define algebras over noncommutative division rings, so $M_n(D)$ is a ring and a vector space over $D$, but not an algebra over $D$.)

(c) Being left  artinian is  a rather special property of a ring. For example, the ring $\ZZ$ has an infinite strictly descending chain of ideals $\ZZ\supsetneq 2\ZZ \supsetneq 4\ZZ \supsetneq 8\ZZ  \supsetneq \dots $
\end{examples}

\begin{remark}
    The ring $M_n(D)$ is thus simple and left artinian. The Wedderburn-Artin Theorem states that the converse of this observation is true.
\end{remark}
\begin{remark}
A right artinian ring is defined analogously via right ideals. However, we will consider only left artinian rings. This is just a matter of choice, the treatment of right artinian rings is essentially no different. In particular,
we can substitute right for left in Theorem \ref{AW}. (A non-simple left artinian ring, however, is not always right artinian.)   
\end{remark}

\subsection{Idempotents}  An element $e$ of a ring  $R$ is called an {\bf idempotent} if $e^2=e$.
Idempotents
$e$ and $f$ are said to be {\bf orthogonal}
if $ef=fe=0$. 
In this case, $e+f$ is  also an idempotent. 

\begin{examples} (a) The simplest pair of orthogonal idempotents is $0,1$. An idempotent different from
$0$ and $1$ is called a {\bf nontrivial idempotent}.

(b)
 If $e$ is an idempotent, then so is $1-e$, and 
 $e$ and $1-e$ are orthogonal.

   (c) The standard matrix units  $E_{11},E_{22},\dots, E_{nn}$ are pairwise orthogonal idempotents.
\end{examples}

\begin{lemma}\label{l21}
If $e$ and  $f\ne 0$ are  orthogonal idempotents in a ring $R$, then $R(1-e)\supsetneq R(1-e-f)$.   
\end{lemma}

\begin{proof}
    Note  that $(1-e-f)(1-e)= 1-e-f$, and hence  $$x(1-e-f)=x(1-e-f)(1-e)\in R(1-e)$$ for every $x\in R$. This proves that  $R(1-e)\supseteq R(1-e-f)$. 
    
 We have  $f=f(1-e)\in R(1-e)$, while
    $f=x(1-e-f)$ with $x\in R$ implies $$0=f(1-f)=x(1-e-f)(1-f)= x(1-e-f) = f,$$ a contradiction. Therefore,
    $R(1-e)\ne R(1-e-f)$.
\end{proof}

If $e$ is an idempotent in a ring $R$, then 
$eRe$ is clearly a ring with unity $e$. 

\begin{lemma}\label{l23} Let $e$ be an idempotent in a ring $R$.
\begin{enumerate}
    \item [{\rm (a)}] If $R$ is left artinian, then so is $eRe$.
    \item [{\rm (b)}] If $R$ is prime, then so is $eRe$.
\end{enumerate}
\end{lemma}

\begin{proof} (a)
    Let $L_1\supseteq L_2 \supseteq L_3\supseteq\dots$ be a  chain of left ideals of $eRe$. Then $RL_1\supseteq RL_2 \supseteq RL_3\supseteq\dots$
    is  a chain of left ideals of $R$. Since $R$ is left artinian, there exists an $n_0\in \NN$ such that
    $RL_n = RL_{n_0}$,
    and hence  $eRL_n = eRL_{n_0}$, for every $n\ge n_0$. However, since $L_n$ is  
    a left ideal of $eRe$,  $eRL_n = L_n$ for each $n$. 
    Therefore, $L_n = L_{n_0}$ for  $n\ge n_0$.

    (b) We verify that 
    $eRe$ satisfies condition (ii) of Lemma \ref{defprime}. Let $eae,ebe\in eRe$ be such that $(eae)eRe(ebe)=\{0\}$. 
    Since $R$ is prime, it follows that    $eae=(eae)e=0$ or $ebe=e(ebe)=0$, as desired.
\end{proof}

\subsection{Minimal Left Ideals}

  A nonzero left ideal $L$ of a ring $R$ is called a {\bf minimal left ideal} if $L$ does not properly contain a  nonzero left ideal of $R$. That is to say, if $I$ is a left ideal such that $I\subseteq L$,
  then $I=\{0\}$ or $I=L$.

\begin{example} \label{ex332}
Let $R=M_n(D)$ with $D$ a division ring, and let $L$ be the set of matrices in $R$ that have 
arbitrary entries in  the $i$th column and zeros in all other columns.
 It is an easy exercise to show that $L$ is a minimal left ideal of $R$. Observe that 
 $L = RE_{ii}$ where $E_{ii}$ is the standard matrix unit, and   $E_{ii}RE_{ii} = \{dE_{ii}\,|\, d\in D\}$, so $E_{ii}RE_{ii}\cong D$. 
\end{example}

\begin{lemma} \label{LD}
Let  $L$ be a  minimal left ideal of a ring $R$. If  $L^2\ne \{0\}$,  then   there exists an idempotent $e\in L$ such that $L=Re$ and 
$eR e$ is  a division ring.
\end{lemma}

 \begin{proof} By assumption, 
there exists a $y\in L$ such that
 $L y\ne \{0\}$. As $L y$ is  a left ideal of $R$ contained in $L$,  $L y =L$ since $L$ is minimal. Therefore, there exists an $e\in L$ such that $ey =y$. Hence, 
 $e^2 y = ey$, meaning that 
 $e^2-e$ lies in the set $ J = \{z\in L \,|\, zy=0\}.$ Observe that $ J$ is a left ideal of $R$ contained in $L$.  Since $Ly\ne \{0\}$ and so $J\ne L$,
 it follows  that $J =\{0\}$. In particular, $e^2=e$. We have  $R e\subseteq L$ 
since $e\in L$, and since $0\ne e\in R e$ it follows from the minimality assumption that
 $L=R e$.
 
   Now consider the ring $eR e$. Let $a\in R$ be such that $eae\ne 0$. We must prove   that $eae$ is invertible in $eR e$. We have $\{0\}\ne R eae \subseteq R e =L$, and so $R eae =L$. Hence, $beae = e$  for some  $b\in R$, which gives $(ebe)(eae)=e$. Since $ebe$ is  a nonzero element in $eR e$, by the same argument there exists a $c\in R$ such that $(ece)(ebe)=e$. But a left inverse  coincides with a right inverse, so  $eae = ece$ is invertible in $eRe$, with inverse $ebe$.
 \end{proof}

\begin{example}
Not every ring has a minimal left ideal. For example, the nonzero ideals of 
the ring $\ZZ$ are $n\ZZ$, $n\in\NN$, and none of them is minimal  since, say, $2n\ZZ\subsetneq n\ZZ$.\end{example}

\begin{lemma}\label{leftmi}
    A nonzero left artinian ring $R$ has a minimal left ideal.
\end{lemma}

\begin{proof}
Choose a nonzero left ideal $L_1$ of $R$ (e.g., $L_1=R$). 
 If $L_1$ is not minimal,  there exists a nonzero left ideal $L_2$ such that  $L_1\supsetneq L_2$. If $L_2$ is not minimal, then  $L_2\supsetneq L_3$ for some nonzero left ideal $L_3$.
Continuing this process we arrive  in a finite number of steps at a minimal left ideal; for if not, there would exist an infinite chain   $L_1\supsetneq L_2\supsetneq L_3\supsetneq  \dots$ 
\end{proof}

\subsection{Matrix Units}
Above, we defined the standard matrix units $E_{ij}$. Now we will consider  their abstract generalization.   
   
     Let $R$ be a ring and let $n\in\NN$. We call $\{e_{ij}\in R \,|\,1\le i,j\le n\}$  a {\bf set of $n\times n$ matrix units} \index{matrix units} if  
     $$e_{11}+e_{22}+\dotsb+e_{nn} = 1\quad\mbox{and}\quad
     e_{ij}e_{kl}=\delta_{jk}e_{il}$$
      for all $1\le i,j,k,l\le n$.  Here, $\delta_{jk}$ stands for the  Kronecker delta.

\begin{example}
    By taking the identity matrix for $A$ in the displayed formula in Example \ref{ex21}\,(b) we see that the standard matrix units $E_{ij}$ are indeed matrix units of $M_n(D)$. There are others. For example, if $S\in M_n(D)$ is invertible,  then $SE_{ij}S^{-1}$ are also matrix units.
\end{example}

  \begin{remark} \label{reim}If $e_{ij}$ are matrix units,
then each ring $e_{ii}Re_{ii}$ is isomorphic to $e_{11}Re_{11}$. 
   Indeed, it is an easy exercise to  verify that
  $$e_{ii}ae_{ii}\mapsto e_{1i}(e_{ii}ae_{ii})e_{i1} = e_{11}(e_{1i}ae_{i1})e_{11}$$
  is  an isomorphism.  Although not directly used in our proof of the Wedderburn-Artin Theorem, this helps in  understanding it better.
 In particular, it explains why the next lemma involves $e_{11}Re_{11}$.
  \end{remark}

\begin{lemma}\label{LMU}
If  a  ring $R$ contains a set of $n\times n$ matrix  units  $e_{ij}$, then $R\cong M_n(e_{11}R e_{11})$.
\end{lemma}

\begin{proof} For every $a\in R$, write $a_{ij} = e_{1i}ae_{j1}$. Observe that $a_{ij} = e_{11}a_{ij} e_{11} \in e_{11}R e_{11}$. 
 Define $\varphi:R\to M_n(e_{11}R e_{11})$ by 
 $$\varphi(a)=(a_{ij}).
 $$
 Our goal is to show that $\varphi$ is an isomorphism.
 
It is clear that $\varphi(a+b)=\varphi(a)+\varphi(b)$ for all $a,b\in R$. The $(i,j)$ entry of $\varphi(a)\varphi(b)$ is equal to 
$$\sum_{k=1}^n e_{1i}ae_{k1}e_{1k}be_{j1} =e_{1i}a\Bigl(\sum_{k=1}^n e_{kk}\Bigr)be_{j1} = e_{1i}abe_{j1},$$
 which is the $(i,j)$ entry of $\varphi(ab)$. Therefore, $\varphi(ab)=\varphi(a)\varphi(b)$. If $a_{ij}=0$ for all $i,j$, then $e_{ii}ae_{jj} = e_{i1}a_{ij}e_{1j} =0$, and so $a=0$ since the sum of all $e_{ii}$ is $1$. Thus, $\varphi$ is injective. Finally, 
 observe that $\varphi(e_{k1}ae_{1l})$ is  the matrix whose $(k,l)$ entry is $e_{11}ae_{11}$ and all other entries are $0$. Since every matrix 
 in $M_n(e_{11}R e_{11})$ is a sum of such matrices, $\varphi$ is surjective.
\end{proof}

Lemma \ref{LMU} tells us that
in order to show that a ring $R$ is a  matrix ring, it is enough to find matrix units in $R$. The next lemma simplifies this  task.

\begin{lemma}
   \label{fajnr}
If a  ring $R$ contains
pairwise 
 orthogonal idempotents $e_{ii}$, $i=1,\dots,n$,  whose sum is  $1$,
 and elements $e_{1i}\in e_{11}Re_{ii}$ and $e_{i1} \in e_{ii}Re_{11}$, $i=2,\dots,n$,  satisfying $e_{1i}e_{i1} = e_{11}$
 and $e_{i1}e_{1i} = e_{ii}$, 
  then  the set $\{e_{ii}, e_{1i},e_{i1}\,|,i=1,\dots,n\}$ can be extended to a set of $n\times n$ matrix units. Accordingly, $R\cong M_n(e_{11}R e_{11})$.
\end{lemma}

\begin{proof}
For $i\ne j$ and $i,j\ne 1$,  define $e_{ij} = e_{i1}e_{1j}$.
Observe that our assumptions imply that  $e_{ij} = e_{i1}e_{1j}$  holds  in any case, that is, also if $i=j$
or one of $i,j$ is $1$. Note also  that
 $e_{1j}e_{k1}=\delta_{jk}e_{11}$ holds
 for all $j$ and $k$.
 Consequently, for all 
 $i,j,k,l$ we have
 $$e_{ij}e_{kl} = e_{i1}e_{1j}e_{k1}e_{1l} = \delta_{jk}e_{i1}e_{11}e_{1l} =\delta_{jk}e_{i1}e_{1l} = \delta_{jk}e_{il},$$
proving that $e_{ij}$ are matrix units.  Therefore, $R\cong M_n(e_{11}R e_{11})$ by Lemma
\ref{LMU}.
\end{proof}

We need two more lemmas.

\begin{lemma}\label{finallem}
    Let $e, f$ be a pair of orthogonal idempotents in a prime ring $R$ such that $eRe$ and $fRf$ are division rings. Then:
    \begin{enumerate}
        \item[{\rm (a)}] There exist $u\in eRf$ and $v\in fRe$ such that $uv=e$ and $vu=f$.
        \item[{\rm (b)}] $(e+f)R(e+f)\cong M_2(eRe)$.
    
        \item[{\rm (c)}] $eRe\cong fRf$. 
    \end{enumerate}
\end{lemma}

\begin{proof}
    (a) Using condition (ii) of
Lemma \ref{defprime} twice we infer that $eafbe\ne 0$ for some $a,b\in R$. As $eRe$ is a division ring with unity $e$, there is a $c\in R$ such that
$(eafbe)(ece)=e$. Thus, $u=eaf\in eRf$ and $v=fbece\in fRe$  satisfy $uv=e$.
Suppose $vu\ne f$. Then $vu-f$ is a nonzero, and hence invertible element of the division ring $fRf$. 
Observe that $vuv = ve =v$ and therefore
$(vu-f)v =0$.
Multiplying this equation
from the left by the inverse of $vu-f$, we arrive at  a contradiction $v= fv=0$. Thus, $vu=f$.

(b) Observe that $e_{11}=e$, $e_{12}=u$, $e_{21}= v$, $e_{22}= f$ form a set of $2\times 2$ matrix units of $(e+f)R(e+f)$, so we can apply Lemma \ref{LMU}.

(c) Use Remark \ref{reim}.
\end{proof}

We  actually only need  the assertion (a). However, (b) and
(c) give a more clear picture. We also remark that  we can replace the assumption that $fRf$ is a  division ring by a milder assumption that $f\ne 0$ and $fRf$ has no nontrivial idempotents. This is because $uv=e$ with $u\in eRf$ and $v\in fRe$ implies $(vu)^2= v(uv)u=veu=vu$. 

\begin{lemma}\label{lemaresz}
    If a prime ring $R$ contains pairwise orthogonal idempotents $e_1,\dots,e_n$ such that 
    their sum is $1$ and $e_iRe_i$ is a division ring for each $i$, then $R\cong M_n(e_1Re_1)$.
\end{lemma}

\begin{proof}
     Write $e_{ii}$  for $e_i$. Applying
Lemma \ref{finallem}\,(a) to the idempotents $e_{11}$ and $e_{ii}$, $i\ge 2$, we see that there exist
elements $e_{1i}\in e_{11}Re_{ii}$ and 
$e_{i1}\in e_{ii}Re_{11}$  satisfying $e_{1i}e_{i1} =e_{11}$ and $e_{i1}e_{1i} =e_{ii}$. Therefore,
$R\cong M_n(e_{11}Re_{11})$ by Lemma \ref{fajnr}.
\end{proof}

\section{Proof of the Wedderburn-Artin Theorem}\label{s3}

The standard, classical version of the 
Wedderburn-Artin Theorem was stated in  Introduction. We will prove a  more general version in which the assumption that $R$ is simple is replaced by the assumption that $R$ is prime.  
The main reason for this is that  prime rings are  more suitable for our proof than simple rings. On the other hand, it is interesting in its own right that nonzero prime left artinian rings are automatically simple.

\begin{theorem}\label{W1}   If $R$ is a nonzero prime left artinian ring, then there exist a division ring $D$ and a positive integer $n$ such that $R\cong M_n(D)$.
\end{theorem}

\begin{proof} 
Since $R$ is left artinian, 
 $R$ contains a minimal left ideal $L$  (by Lemma \ref{leftmi}), and since $R$ is prime, $L^2\ne \{0\}$ (by Lemma \ref{defprime}). Lemma \ref{LD} therefore implies that $R$ contains an idempotent $e_1$ such that $e_1Re_1$
is a division ring. By Lemma \ref{l21} (applied to $e=0$),
$R\supsetneq R(1-e_1)$. If $e_1\ne 1$,
 then $(1-e_1)R(1-e_1)$ is a nonzero ring, which is also prime and left artinian by Lemma \ref{l23}. We can thus repeat  the argument from the beginning of the proof to conclude  that $(1-e_1)R(1-e_1)$ contains an idempotent $e_2$  such that
 $e_2(1-e_1)R(1-e_1)e_2$
 is a division ring. Observe that $e_2\in (1-e_1)R(1-e_1)$ implies that $e_1$ and $e_2$ are orthogonal. Therefore, 
  $e_2(1-e_1)R(1-e_1)e_2 =e_2Re_2$
 and
 $R(1-e_1)\supsetneq R(1-e_1-e_2)$  by Lemma \ref{l21}. If $e_1+e_2\ne 1$,  the same argument shows that there exists an
 idempotent $e_3$ belonging to $(1-e_1-e_2)R (1-e_1-e_2) $, and hence orthogonal to $e_1$ and $e_2$, such that 
  $e_3(1-e_1-e_2)R(1-e_1-e_2)e_3 =e_3Re_3$
 is a division ring and $R(1-e_1-e_2)\supsetneq R(1-e_1-e_2-e_3)$. We continue this process, which must stop after finitely many steps since $R$ is left artinian. Therefore, there exist pairwise orthogonal idempotents $e_1,\dots,e_n\in R$
such that $e_iRe_i$ is a division ring and $e_1+\dots+e_n=1$. Lemma \ref{lemaresz} now gives the desired conclusion that $R\cong M_n(e_1Re_1)$.
  \end{proof}

\end{document}